\documentclass[a4paper,english,oneside]{amsart}
\usepackage[utf8]{inputenc}
\usepackage[T1]{fontenc}
\usepackage[hscale=0.68, vscale=0.7]{geometry}
\usepackage[numbers]{natbib}
\usepackage{amssymb}
\usepackage{mathtools}

\DeclareFontFamily{OMX}{mlmex}{}
\DeclareFontShape{OMX}{mlmex}{m}{n}{%
   <->mlmex10%
   }{}%
\usepackage{mlmodern}

\theoremstyle{plain}
\newtheorem{theorem}{Theorem}
\newtheorem{proposition}[theorem]{Proposition}
\newtheorem{lemma}[theorem]{Lemma}

\newtheorem{remark}{Remark}

\theoremstyle{definition}

\usepackage{color}

\usepackage{hyperref}
\hypersetup{%
  colorlinks=true,%
  citecolor=[RGB]{120,29,126},
  pdfauthor={Jean-François Burnol},%
  pdfsubject={Ellipsephic numbers, Kempner and Irwin series},%
  pdfstartview=FitH,%
  pdfpagemode=UseNone,%
}
\usepackage{breakurl}

\DeclareMathOperator{\ord}{ord}

\newcommand\plusbas{\vphantom{X^X}}

\newcommand\cA{\mathcal{A}}

\newcommand\sD[1][]{%
  \if\relax\detokenize{#1}\relax\else{}_{\plusbas#1}\mskip-1mu\relax\fi
  \mathsf{D}%
}
\newcommand\bsD[1][]{%
  \if\relax\detokenize{#1}\relax\else{}_{\plusbas#1}\mskip-1mu\relax\fi
  \boldsymbol{\mathsf{D}}%
}

\newcommand\dx{\mathrm{d}\mskip-1mu x}

\allowdisplaybreaks

\usepackage{setspace}

\title[Ellipsephic harmonic series]{%
  Measures associated with certain ellipsephic harmonic series
  and the Allouche-Hu-Morin limit theorem}

\author[J.-F. Burnol]{Jean-François Burnol}

\address{Université de Lille,
  Faculté des Sciences et technologies,
  Département de mathématiques,
  Cité Scientifique,
  F-59655 Villeneuve d'Ascq cedex,
  France}
\email{jean-francois.burnol@univ-lille.fr}

\date{v2 January 1st, 2025. The ms has since been accepted for publication in
  Acta Mathematica Hungarica. Differs from v1 May 6, 2024 via a new title,
  some minor notational changes and a few bibliographical updates and other
  improvements.}


\begin{document}

\begin{abstract}
  We consider the harmonic series $S(k)=\sum^{(k)} m^{-1}$ over the integers
  having $k$ occurrences of a given block of $b$-ary digits, of length $p$,
  and relate them to certain measures on the interval $[0,1)$.  We show that
  these measures converge weakly to $b^p$ times the Lebesgue measure, a fact
  which allows a new proof of the theorem of Allouche, Hu, and Morin which
  says $\lim S(k)=b^p\log(b)$.  A quantitative error estimate will be given.
  Combinatorial aspects involve generating series which fall under the scope
  of the Goulden-Jackson cluster generating function formalism and the work of
  Guibas-Odlyzko on string overlaps.
\end{abstract}

\maketitle

\onehalfspacing

\section{Introduction}

Throughout this work, an integer $b>1$ and a block of $b$-ary digits
$w=d_1\dots d_p$ having length $|w|=p\geq1$ are fixed.

Let $S_w(k)= \sum^{(k)}m^{-1}$ where the denominators are the positive
integers whose (minimal) representations as strings of $b$-ary digits contain
exactly $k$ occurrences of the block $w$.  Their finiteness will be reproven
in the text body.  For single-digit blocks $w$, i.e.\@ $p=1$, Farhi
\cite{farhi} proved $\lim S_w(k)=b \log(b)$.  Allouche, Hu, and Morin
\cite{allouchehumorin2024} (and \cite{allouchemorin2023} earlier for the
$b=2$ case) have now extended this to all $p\geq1$:
\begin{equation}\label{eq:ahm}
  \lim_{k\to\infty} S_w(k) = b^p\log(b)\;.
\end{equation}
The approach of \cite{allouchemorin2023,allouchehumorin2024} exploits special
properties of a certain rational function of a combinatorial nature, which had
been defined and studied for binary base already in \cite{alloucheshallit1989,
  allouchehajnalshallit1989}.  One of the reasons why the validity of
\eqref{eq:ahm} for $p>1$ appears to be significantly more difficult to
establish is the possibility of self-overlaps (strings shorter
than $ww$ may already contain three occurrences or more of the word $w$).

Using completely different tools, we confirm \eqref{eq:ahm} in the following
quantitative form:
\begin{theorem}\label{thm:limit}
  For each $k\geq1$, $\Bigl| S_w(k) - b^{|w|} \log(b)\Bigr|\leq (b-1)b^{|w|-\max(k,2)+1}$.
\end{theorem}
This estimate is definitely not sharp; the author studied in
\cite{burnolasymptotic} the situation for one-digit blocks $w$, with $k$ fixed
and $b\to\infty$ and the results suggest, 
for $|w|=1$, that the distance of $S_w(k)$ to $b\log(b)$ is of the order
$b^{-2k}$ up to some factor depending on $b$ and the digit $w$.
\begin{remark}
\singlespacing
  The rules of
  \cite{alloucheshallit1989,huyining2016,allouchemorin2023,allouchehumorin2024}
  for counting (possibly overlapping) occurrences of $w$ in an integer $n$ are special
  if $w$'s first digit is $0$ and not all its digits are $0$: the integer $n$
  is first extended to the left with infinitely many zeros.  This either will
  not modify the number of occurrences or will increase it by exactly one
  unit.  Due to this, our $S_w(k)$ for such $w$'s is not the same quantity as
  considered in \cite{allouchehumorin2024}.  With $S_w(k)$ replaced by the one
  using the conventions in \cite{allouchehumorin2024}, a modified upper bound
  holds: our method obtains Theorem \ref{thm:limit} but with an extra $b$
  factor and the requirement $k\geq2$.
\par
\end{remark}

In \cite{burnolirwin} where we investigated the $S_w(k)$'s for $w$ a single-digit
block (Irwin series \cite{irwin}), we worked with formulae of the type
\begin{equation}\label{eq:loglike}
  S_w(k) = \int_{[b^{-1},1)} \frac{\mu_k(dx)}{x}\;,
\end{equation}
for some discrete measures $\mu_k$ on $[0,1)$.
In \cite{burnolone42} this was extended to the $|w|=2$ case.  In
\cite{burnolirwin,burnolone42}, the integral formula \eqref{eq:loglike} is
submitted to transformations induced from recurrence properties among the
$\mu_k$'s, and then converted, using the moments of the $\mu_k$'s, into series
allowing the numerical computation of the $S_w(k)$'s.

Here, we will not transform the integral formula  \eqref{eq:loglike} any
further but concentrate on what can be derived directly from it as $k$ goes
to $\infty$.  The essential fact which we establish is:
\begin{theorem}\label{thm:conv}
  The measures $\mu_k$ converge weakly to $b^p\dx$, i.e. for any interval
  $I\subset [0,1)$, there holds $\lim_{k\to\infty}  \mu_k(I) = b^p |I|$, where
  $|I|=\sup I - \inf I$ is the Lebesgue measure of $I$.
\end{theorem}
This explains the Farhi and Allouche-Hu-Morin theorem \eqref{eq:ahm}, in a
manner very different from the works of these authors.  Further, for intervals
$I=[t,u)$, with $t<u$ rational numbers with powers of $b$
denominators (``$b$-imal numbers''), the sequence $(\mu_k(I))$ becomes
constant for $k$ large enough, which is what will give the quantitative estimate
\ref{thm:limit}.

The contents are as follows: the next section will explore underlying
combinatorics and obtain the core result for our aims, which is Theorem
\ref{thm:totalmass}.  A general framework of Goulden-Jackson
(\cite{gouldenjackson1979}, \cite[\S2.8]{gouldenjacksonwiley}) could be
applied to obtain the shape of certain generating series which are parts of
the input to Theorem \ref{thm:totalmass}, but we shall work from first
principles.  Although it is not needed for our aims, a result of
Guibas-Odlyzko \cite{guibasodlyzko1981b} will also be proven.

Then in the last section we derive Theorems \ref{thm:limit} and \ref{thm:conv} as easy
corollaries.

\section{Generating series for 
         the \texorpdfstring{$w$}{w}-block counts}

It has definite advantages in the study of the $S_w(k)$'s to work not only
with integers but also with strings (they may also be called ``words''),
i.e.\@ elements of $\bsD=\cup_{l\geq0}\sD^l$, $\sD = \{0, \dots ,b-1\}$.  So
all strings $X$ considered here have a finite length $l=|X|$. The empty string
shall be denoted $\epsilon$, and $|\epsilon|=0$.  The length $|X|$ may also be
denoted $l(X)$.

Each such string of $b$-ary digits $X$ defines a non-negative integer $n(X)$
in the usual way of positional left to right notation.  We set $n(\epsilon) =
0$.  Conversely each non-negative integer $n$ has a minimal representation
$X(n)$ as a string, all others being with prepended leading zeros. In particular
$X(0)=\epsilon$.

For any string $X$ we let $k_w(X)$ be the number of (possibly overlapping) occurrences
of $w$ in $X$.
We say that a string is $k$-admissible if $k_w(X) = k$.  Let $N_w(k,l)$ for
each $k\geq 0$ and $l\geq 0$ be the number of $k$-admissible strings of
length $l$.  We let $Z_w(k)$ be the corresponding generating series in the
indeterminate $t$:
\begin{equation}
  Z_w(k) = \sum_{l=0}^\infty N_w(k,l) t^l\;.
\end{equation}
Whenever we subtitute $b^{-1}$ for $t$ in such a generating series, we say
that we compute the ``total mass'' of the objects counted by the series.
I.e.\@ each string $X$ is weighted as $b^{-|X|}$.  For example strings of
length $l$ have a total mass equal to $1$ (also for $l=0$, there is then only
one $X$, the empty string). Here is our main result:
\begin{theorem}\label{thm:totalmass}
  Let $s$ be any string of length $\ell$, and let $k_w^*(s) = \max k_w(sz)$
  where the maximum is taken over all strings $z$ of length $p-1$ (with
  $p=|w|$).  For each $k\geq 1 + k_w^*(s)$, the total mass of the $k$-admissible
  strings having $s$ as prefix is equal to $b^{p}\cdot b^{-\ell}$.  In
  particular for any $k\geq1$, $Z_w(k)(b^{-1})=b^p$.
\end{theorem}
Note that $k_w^*(s)\leq \ell$ in the above theorem so the conclusion applies to
any $k>\ell$.  As notational shortcut, in place of appending $(b^{-1})$ to
express evaluation at $b^{-1}$, we sometimes replace the letter $Z$ by $M$
(for ``mass''), so the theorem says in particular $M_w(k)=b^p$.

It turns out that the arguments which allowed us to prove Theorem
\ref{thm:totalmass} also allow to express all $Z_w(k)$'s, $k\geq1$, in
terms of $Z_w(0)$.  The existence of such expressions is
related to the fact that the doubly generating series
\begin{equation}
  \sum_{k=0}^\infty \sum_{l=0}^\infty N_w(k,l)r^kt^l = \sum_{k=0}^\infty Z_w(k)r^k\;,
\end{equation}
falls under the scope of a general theory of Goulden and Jackson (see
\cite{gouldenjackson1979}, \cite[\S2.8]{gouldenjacksonwiley}; see also
Exercise 14 from \cite[Chap. 4]{stanleyI} as an introduction) where a notion
of \emph{cluster generating function} is fundamental.  It would take us a bit
too far to discuss the details of the notion here, and is not needed as
we will obtain a complete description of the $Z_w(k)$ from first principles.

Guibas and Odlyzko define in \cite{guibasodlyzko1981b} the
\emph{auto-correlation polynomial} $A_w$ of a block $w$ as follows: $A_w =
\sum_{0\leq i<p} c_i t^i$ with $c_i=1$ if the prefix of $w$ of length $p-i$ is
identical with the suffix of the same length, else $c_i=0$.  Thus the indices
$i$ with $c_i=1$ are the \emph{periods} of $w$, i.e., there holds $d_{j} =
d_{i+j}$ for $1\leq j\leq p-i$.  For example, with $p=3$, $A_w=1+t+t^2$ for
$w=aaa$, $1+t^2$ for $w=aba$ with $b\neq a$, and $1$ else.  For all $p$'s, if
$c_1=1$ then $w$ contains only the same repeated letter hence all other
$c_i$'s have to be also $1$.  In \cite{guibasodlyzko1981a}, Guibas and
Odlyzko have characterized the polynomials realizable as such an $A_w$. In
\cite{guibasodlyzko1981b} they prove:
\begin{equation}\label{eq:go}
  Z_w(0) = \frac{A_w}{(1 - bt)A_w + t^p}\;.
\end{equation}
It turns out that we do not really need formula \eqref{eq:go}, but for
completeness we will include a proof.

The rationality of $Z_w(0)$ can also be obtained as a very special case of
Proposition 4.7.6 of \cite{stanleyI} which itself is an application of the
general \emph{transfer matrix} method.

Let us start with proving that  $b^{-1}$ is always in the open disk of
convergence.  This is most probably in the literature but we do not
know a reference.
\begin{lemma}
  Each $Z_w(k)$ has a radius of convergence at least equal to
  $(b^p-1)^{-p^{-1}}$.  And
  \begin{equation}\label{eq:upper}
    \sum_{0\leq j \leq k} Z_w(j)(b^{-1}) \leq p (k+1) b^p\;.
  \end{equation}
\end{lemma}
\begin{proof}
  Let us obtain an upper bound for $\sum_{0\leq j\leq k} N_w(j,l)$.  We write
  $l=qp + r$ with $0\leq r< p$.  We consider a string of length $l$ 
  as $q$ contiguous blocks each of length $p$ and $r$ extra digits.  If such a
  string of length $qp+r$ has at most $k$ occurrences of $w$, then in
  particular among the $q$ blocks of length $p$, at most $k$ of them are equal
  to $w$. For each $i$, $0\leq i\leq k$, there are $\binom{q}{i}(b^p-1)^{q-i}$
  ways (with $\binom{q}{i}=0$ for $i>q$) of obtaining a length $qp$
  string having exactly $i$ among its $q$ contiguous length $p$ blocks
  which are equal to $w$.  It is possible for such a string to have $>k$
  occurrences of $w$, but all strings of length $qp$ and $\leq k$ occurrences
  of $w$ are accounted for. Hence
  \begin{equation}
    \sum_{0\leq j\leq k} N_w(j,l) \leq \sum_{0\leq i\leq k}
    \binom{q}{i}(b^p-1)^{q-i}\times b^r
, \qquad l = q p + r\;, 
  \end{equation}
  where $b^r$ is for the remaining, unconstrained, $r$ digits. We obtain in
  this way a dominating series, i.e.\@ a series whose coefficients are at
  least equal to the (absolute values of the) original ones in the termwise
  addition $\sum_{0\leq j \leq k} Z_w(j)$ (which has only non-negative
  coefficients):
  \begin{equation}
    \sum_{0\leq i\leq k}
    \sum_{q=0}^\infty \binom{q}{i}(b^p-1)^{q-i} t^{pq}\times \sum_{0\leq r < p} b^r t^r\;.
  \end{equation}
  We recognize therein binomial series and reformulate the dominating series as
  \begin{equation}
    \sum_{0\leq i\leq k}
    \frac{t^{pi}}{(1 - (b^p-1)t^p)^{i+1}}\times \sum_{0\leq r < p} b^r t^r\;.
  \end{equation}
  The radius of convergence of $\sum_{l=0}^\infty (\sum_{0\leq j\leq k}
  N_w(j,l))t^l$ is at least the one of the dominating series, which from the
  above is $R=(b^p-1)^{-p^{-1}}$.  Hence the radii of convergence of the
  individual series $Z_w(j)=\sum_{l=0}^\infty N_w(j,l)t^l$ with non-negative
  coefficients all are $\geq R$.  This $R$ is larger than $b^{-1}$, and
  substituting $t= b^{-1}$ leads to the upper bound \eqref{eq:upper}.
\end{proof}
\begin{lemma}
  For all $k\geq0$, $S_w(k) < \infty$.
\end{lemma}
\begin{proof}
  Consider the contribution to $S_w(k)$ by integers of a given length
  $l\geq1$, i.e. those integers in $[b^{l-1},b^l)$ having $k$ occurrences of
  $w$.  This contribution is bounded above by $b^{1-l}N_w(k,l)$, hence
  $S_w(k)\leq b Z_w(k)(b^{-1}) < \infty$.
\end{proof}
We have defined the generating series $Z_w(k)$. Let us similarly define
$Z_w(x,j,y)$ as the generating series of the counts per length of
$k$-admissible strings which start with a given block $x$ and end with a given
block $y$.  We will mainly use them for $x$ and $y$ being both of length
$p-1$.  So if $p=1$, $x$ and $y$ will then both be the empty string $\epsilon$
and add no constraints.  Let us further define similarly $Z_w(x,j)$ as a
shortcut for $Z_w(x,j,\epsilon)$, i.e.\@ the counted $j$-admissible strings
have to start with $x$ and no condition on how they end.  We will also need
$Z_w(\epsilon,j,y)$.  We can not denote it $Z_w(j,y)$ as the two-argument $Z_w$
is already defined.

We let $u$ be the  prefix (initial sub-block) $d_1\dots d_{p-1}$ of $w$
and $v$ its suffix (terminating sub-block) $d_2\dots d_{p}$.
\begin{lemma}\label{lem:a}
  Let $k\geq1$.  Then for any $(p-1)$-block $s$ one has
  \begin{equation}
    Z_w(s,k) = t^{2-p} Z_w(s,0,u)Z_w(v,k-1)\;.
  \end{equation}
\end{lemma}
\begin{proof}
  Cut the $k$-admissible string starting with $s$ into two pieces: the first
  one ends with the $u$ part of the left-most occurrence of $w$.  The second
  piece starts with the $v$ part of this left-most occurrence of $w$.  This
  second piece $R$ has exactly $k-1$ occurrences of $w$.  The first piece $L$
  has none.  Conversely any two such pieces $L$ and $R$ can be glued together
  along their $p-2$ common digits (if $p=1$, this means creating $LwR$, if
  $p=2$ it is simply concatenation $LR$, if $p>2$, it is gluing with the
  $p-2$ digits overlap).  So we have a Cauchy product of generating functions
  which we must correct by the factor $t^{2-p}$ for matters of total
  lengths.
\end{proof}
Using this lemma with $s=v$ the $(p-1)$-suffix of $w$ we obtain in particular,
for all $k\geq0$:
  \begin{equation}\label{eq:Zvk}
    Z_w(v,k) = \Bigl(t^{2-p} Z_w(v,0,u)\Bigr)^{k} Z_w(v,0)\;.
  \end{equation}
To shorten notation, when evaluating at $b^{-1}$, in place of appending
$(b^{-1})$ we replace the letter $Z$ by $M$ (for total mass).
\begin{lemma}\label{lem:b}
  $M_w(v,k) = Z_w(v,k)(b^{-1})$ is a constant which is independent of $k\geq0$.
\end{lemma}
\begin{proof}
  From \eqref{eq:Zvk}, $M_w(v,k)=\alpha \beta^k$, for some
  $\alpha,\beta\geq0$.  The hypothesis $\beta>1, \alpha>0$ is incompatible
  with the (linearly divergent, not geometrically) upper bound
  \eqref{eq:upper}.  And the hypothesis $\beta<1$ or $\alpha=0$ is
  incompatible with the fact that $\sum_{k=0}^\infty Z_w(v,k)(b^{-1}) =
  \infty$.  Indeed this sum evaluates the total mass of all strings starting
  with $v$ and weighted by $b^{-l}$ where $l$ is their length, and is thus
  infinite.  So $\beta=1$.
\end{proof}
The next lemma will imply that the constant value of $M_w(v,k)$, $k\geq0$, is $b$.
\begin{lemma}\label{lem:c}
  The following relations hold and show that knowing one of $Z_w(0)$,
  $Z_w(v,0)$ or $Z_w(v,0,u)$ determines the other two.  They also imply that
  $M_w(v,0) = b$ and $M_w(v,0,u) = b^{2-p}$.
  \begin{align}
\label{eq:v0u}
    (1 - bt) t^{1-p}Z_w(v,0) &= 1 - t^{2-p} Z_w(v,0,u)\;,
\\
\label{eq:v0}
    (1 - bt)Z_w(0) &= 1  - t Z_w(v,0)\;.
  \end{align}
\end{lemma}
\begin{proof}
  One surely has $\ord_{0}(t^{2-p}Z_w(v,0,u))\geq 1$, so it is licit in
  the algebra of formal power series to compute the sums of the equations
  \eqref{eq:Zvk} for all $k\geq0$, to obtain the identity
  \begin{equation}
    \frac{t^{p-1}}{1 - bt} 
  = \sum_{k=0}^\infty Z_w(v,k) = \frac{Z_w(v,0)}{1 - t^{2-p} Z_w(v,0,u)}\;,
\end{equation}
hence \eqref{eq:v0u}.  Let $x = dx'$ be a $0$-admissible string, of length at
least $1$.  If $d\neq d_1$ (where $d_1$ is the first digit of $w$), then $x'$
is an arbitrary (possibly empty) $0$-admissible string. If $d=d_1$ then $x=d_1
x'$ with $x'$ an arbitrary (possibly empty) admissible string not starting
with the $(p-1)$-suffix $v$ of $w$.  Hence
  \begin{equation}
    Z_w(0) =  1 + (b-1)t Z_w(0) + t (Z_w(0) - Z_w(v,0))\;,
  \end{equation}
which gives the second equation.

Substituting $t=b^{-1}$ we obtain the stated evaluations of $M_w(v,0)$ and
$M_w(v,0,u)$.
\end{proof}
\begin{lemma}
  The two series $Z_w(\epsilon,0,u)$ and $Z_w(v,0)$ are identical.
\end{lemma}
\begin{proof}
  Let $r(w)$ be the reversed block.  Its $(p-1)$-suffix is $r(u)$.  We thus
  have $(1-bt)Z_{r(w)}(0) = 1 - t Z_{r(w)}(r(u),0)$.  But mapping strings to
  their reversals, it is clear that $Z_{r(w)}(0) = Z_w(0)$, and
  $Z_{r(w)}(r(u),0) = Z_w(\epsilon,0,u)$.  Thus $Z_w(\epsilon,0,u) =
  Z_w(v,0)$.
\end{proof}

\begin{lemma}\label{lem:d}
  For $k\geq1$ one has 
  \begin{equation}\label{eq:Zwk}
    Z_w(k) = t^{2-p}\Bigl(t^{2-p}Z_w(v,0,u)\Bigr)^{k-1} Z_w(v,0)^2\;,
  \end{equation}
  which via equations \eqref{eq:v0u} and \eqref{eq:v0} can be used to express
  $Z_w(k)$ in terms of $Z_w(0)$.
\end{lemma}
\begin{proof}
  For any given block $s$ of length $p-1$ we have from Lemma \ref{lem:a} $
  Z_w(s,k) = t^{2-p} Z_w(s,0,u)\cdot Z_w(v,k-1)$.  For $k\geq1$, a $k$
  admissible string $X$ necessarily has length at least $p$, hence it has a
  prefix $s$ of length $p-1$ available.  So $Z_w(k)$ is the sum of the
  $Z_w(s,k)$ over all $s$ of length $p-1$.  Similarly $\sum_{s, |s|=p-1} Z_w(s,0,u) =
  Z_w(\epsilon, 0, u)$ because any string terminating in $u$ has length at
  least $p-1$ hence has a starting part $s$ of length $p-1$ associated with
  it.  Hence 
  \begin{equation*}
    Z_w(k) = t^{2-p}Z_w(\epsilon,0,u)Z_w(v,k-1)\;,
  \end{equation*}
  and the formula for $Z_w(v,k-1)$ is given by \eqref{eq:Zvk}.  And we
  know that $Z_w(\epsilon,0,u) = Z_w(v,0)$.
\end{proof}
We have all elements needed  to complete our main objective, the proof of
Theorem \ref{thm:totalmass}:
\begin{proof}[Proof of Theorem \ref{thm:totalmass}]
Substituting $t=b^{-1}$ in \eqref{eq:Zwk} we obtain
\begin{equation}
  M_w(k) = b^{p-2}\bigl(b^{p-2}M_w(v,0,u)\bigr)^{k-1}M_w(v,0)^2 =
  b^{p-2}\cdot 1^{k-1} \cdot b^2 = b^p\;.
\end{equation}
Now let $s$ be an arbitrary string and define $k_w^*(s) = \max_{|z|=p-1}k_w(sz)$ as in the
theorem statement.  Suppose $k>k_w^*(s)$ and let $sx$ be an $s$-prefixed string
which is $k$-admissible. Thus $|x|\geq p$.  Let split $x$ into $zy$ with $z$
its prefix of length $p-1$ (and $y$ necessarily not empty).  If $w$ occurs in
$sx = szy$ it is either in $sz$ or in $zy$, so $k_w(sx) = k_w(sz)+k_w(zy)$.  We can
thus partition the set of all $x$ such that $k_w(sx)=k$ according to the
$b^{p-1}$ possible $z$'s.  And we obtain the identity associated with this
partition:
\begin{equation}
  k>k_w^*(s)\implies  Z_w(s,k) = \sum_{|z| = p-1} t^{|s|} Z_w(z, k - k_w(sz))\;.
\end{equation}
According to Lemma \ref{lem:a} and equation \ref{eq:Zvk}, we have
\begin{equation}
  j\geq 1 \implies Z_w(z, j) = t^{2-p}Z_w(z,0,u)\bigl(t^{2-p}Z_w(v,0,u)\bigr)^{j-1}Z_w(v,0)\;.
\end{equation}
Using now the masses given in Lemma \ref{lem:c} (and thanks to $j_z\coloneq k-k_w(sz)>0$) we
obtain
\begin{equation}
  Z_w(s,k)(b^{-1}) = \sum_{|z| = p-1} b^{-|s|}\cdot b^{p-2} \cdot Z_w(z,0,u)(b^{-1})\cdot 1^{j_z-1}\cdot b \;.
\end{equation}
We have explained earlier that $\sum_{|z|=p-1} Z_w(z,0,u) = Z_w(\epsilon,0,u)
= Z_w(v,0)$ and we know that substituting $t=b^{-1}$ gives a mass equal to
$b$.  So $Z_w(s,k)(b^{-1})=b^p b^{-|s|}$.
\end{proof}

We will also need the following additional fact (for $k\geq2$ it is
a corollary of the Theorem \ref{thm:totalmass} we just proved) which for $p=2$ and
$k=1$ had been already observed in \cite[\S3]{burnolone42}:
\begin{proposition}\label{prop:d1}
  For any digit $d\in \sD$, the total mass of the $k$-admissible strings
  starting with $d$ is $b^{p-1}$ if $k\geq1$.  So the total mass of the
  $k$-admissible integers is $(b-1)b^{p-1}$ for $k\geq1$.
\end{proposition}
\begin{proof}
  Let $d_1$ be the first digit of $w$.  Let $d\neq d_1$ be another digit.
  Then $X=dx$ is $k$-admissible if and only $x$ is, hence the result in that
  case.  Thus the $k$-admissible strings not starting with $d_1$ make up for a
  total mass of $(b-1)b^{-1}\cdot b^p$.  This forces the $k$-admissible
  strings starting with $d_1$ to have a total mass $b^{p-1}$ too (we have used
  that the empty string is not admissible as $k\geq1$).

  We can deduce in a novel manner $M_w(v,k)=M_w(v,k-1)$ hence Lemma
  \ref{lem:b} from this.  The above paragraph showed (without using any
  anterior result apart from finiteness) that the mass of $k$-admissible
  strings starting with $d_1$ is $b^{-1}$ times $M_w(k)$.  But $X=d_1x$ is
  admissible if either $x$ does not start with $v$ and is $k$-admissible, or
  $x$ starts with $v$ and is $(k-1)$-admissible.  This gives a mass equal to
  $b^{-1}(M_w(k) - M_w(v,k)) + b^{-1} M_w(v,k-1)$.  But the final result has
  to be $b^{-1}M_w(k)$.  So $M_w(v,k)=M_w(v,k-1)$.
\end{proof}

In order to have a fully-rounded picture we now go the extra step of proving
the Guibas-Odlyzko formula \eqref{eq:go}.  We can partition the set $\cA$ of
$0$-admissible strings $X$ according to the value of $k_w(vX)$.  We obtain a
partition of $\cA$ into $\cA_0$, \dots , $\cA_{p-1}$ (some of these possibly
empty, of course).  Suppose $X$ belongs to $\cA_j$, $j\geq1$ and let
$i\in\{1,\dots,p-1\}$ be the index in $wX$, counting from the left and
starting at $0$, of the last occurrence of $w$ in $wX$.  We can think of $i$
as how many times we need to shift $w$ one unit to the right in order to reach
that final occurrence of $w$ in $wX$.  It is also the index in $vX$ of the
last occurrence of $w$, but starting the index count at $1$, not $0$.  Recall
in what is next that we write $w=d_1\dots d_p$.

The string $X\in\cA_j$ thus starts with the last $i$ digits of $w$: $d_{p-i+1}\dots
d_{p}$ is a prefix of $X$.  And it is necessary that $d_{p-i}$ be the first
digit to the left of $X$ in $wX$, i.e. $d_{p-i}=d_p$, and further
$d_{p-i-1}=d_{p-1}$, etc... until $d_1=d_{i+1}$.  So the index $i\geq1$ is
necessarily a period of $w$.  We will now prove that $i$ is exactly the $j$th
positive period of $w$ in increasing order.

Let $s_i$ be the length $i$ suffix of $w$, and $r_i$ the length $i$ prefix of
$w$, for the period $i$.  Then $ws_i=r_iw$.  We found that $X$ starts with $s_i$,
$X=s_iY$, so $wX= ws_i Y = r_i w Y$.  We have $k_w(wX)=k_w(r_iwY) =
k_w(r_iw)+k_w(wY)-1=k_w(r_iw)+k_w(vY)$.  Thus $j+1 = k_w(wX) = k_w(r_iw) +
k_w(vY)$. We have defined $i$
as giving the position in $wX$ of the last occurrence of $w$, and the $s_i$ in the
decomposition $X=s_iY$ is the suffix of this last occurrence of $w$.  So no
occurrence of $w$ in $wX$ can touch $Y$, which means that $k_w(vY)=0$ and thus
$j+1 = k_w(r_i w)=k_w(ws_i)$.  This characterizes $i$ as being the $j$th positive
period of $w$, indeed we observe that for all smaller periods $\iota$, their
$ws_\iota$ is a prefix of $ws_i$.
Let us thus now denote our $i$ also as $i_j$.

We have also obtained in passing that $X = s_i Y$ with $Y\in \cA_0$.
Conversely for such an $X$ we have $k_w(wX) = k_w(ws_i Y) = k_w(r_iwY)=
k_w(r_iw)+k_w(vY) = k_w(r_iw) = k_w(ws_i) = j+1$.  We can conclude that $\cA_j = s_{i_j} \cA_0$
hence its associated generating series is $t^{i_j}$ times the generating
series associated with $\cA_0$.  But the latter is in our notation
$t^{1-p}Z_w(v,0)$.  Hence, using the auto-correlation polynomial $A_w$, we
reach the key final formula:
\begin{equation}
  Z_w(0) = A_w t^{1-p} Z_w(v,0)\;.
\end{equation}
Combining with Lemma \ref{lem:c}, we obtain the
Guibas-Odlyzko formula \eqref{eq:go}:
\begin{align}
    Z_w(0) &= \frac{A_w}{(1 - bt) A_w + t^p}\;,
\\\shortintertext{and}
    Z_w(v,0)  &= \frac{t^{p-1}}{(1 - bt) A_w + t^p}\;,
\\
    t^{2-p}Z_w(v,0,u) &= \frac{(1 -bt)(A_w - 1) + t^p}{(1 - bt) A_w  + t^p}\;.
  \end{align}
Hence Lemma \ref{lem:d} becomes:
\begin{equation}
  k \geq 1\implies Z_w(k) = t^p \frac{\bigl((1 -bt)(A_w - 1) + t^p\bigr)^{k-1}}{\bigl((1 - bt) A_w  + t^p\bigr)^{k+1}}\;.
\end{equation}
The values $M_w(k)=b^p$ are obvious here.  And we obtain $M_w(0)=b^p A_w(b^{-1})=b^p +
b^{p-i_1} + \dots $ in terms of the positive periods $0<i_1<i_2<\dots$, if any.

\section{Measures and integrals}

For each string $X$, we define the $b$-imal number $x(X)\in [0,1)$ simply by
putting $X$ immediately to the right of the $b$-imal separator, i.e.\@ $x(X) =
n(X)/b^{|X|}$.  We define the measure $\mu_k$ as the sum over all
$k$-admissible strings of the Dirac masses at points $x(X)$ with weights
$b^{-|X|}$.  In \cite[\S4]{burnolirwin} it is explained that as the discrete
measure has finite total mass, one can use it to integrate any bounded or any
non-negative function (in the latter case possibly obtaining $+\infty$).  In
particular, consider integrating the function $g(x)$ being defined as $1/x$
for $b^{-1}\leq x < 1$ and zero elsewhere.  The computation of $\mu_k(g)$
involves only those strings $X$ starting with a non-zero $b$-ary digit (in
particular the empty string which maps to the real number $0$ is not involved)
hence contributions are in one-to-one correspondance with those positive
integers $m$ having exactly $k$ occurrences of the block of digits $w$.  And
the contribution of an integer $m\in [b^{l-1},b^l)$ is ${(m/b^{l})^{-1}}$
(value of the function) times $b^{-l}$ (weight associated with the Dirac
point-mass). That gives exactly $m^{-1}$ hence formula \eqref{eq:loglike} for
$S_w(k)$.

This reasoning works for all $k\geq0$ but now we suppose $k\geq1$. Recall in
the following that $p=|w|$.

Let us first suppose $k=1$.  We obtain from \eqref{eq:loglike} the bounds
\begin{equation}
  \mu_1([b^{-1},1)) \leq S_w(1)
  \leq b \cdot \mu_1([b^{-1},1))\;.
\end{equation}
Now, $\mu_1([b^{-1},1))$ is the total mass (i.e. each string $X$ weighs
$b^{-|X|}$ if kept) of the $1$-admissible strings starting with a non zero
digit.  By proposition \ref{prop:d1} this is $(b-1)b^{p-1}$.  We also have
(rather trivial and sub-optimal) bounds
\begin{equation}
  b^p(1 - b^{-1}) \leq\int_{b^{-1}}^1 \frac{b^p \dx}{x} \leq b\cdot b^p(1 - b^{-1})\;.
\end{equation}
Hence $|b^p \log(b) - S_w(1)|\leq (b-1)^2 b^{p-1}$.
It was not really optimal to replace $1/x$ either by $1$ or by $b$! So let us
make a finer decomposition along sub-intervals $I_a = [a/b,(a+1)/b))$, $1\leq
a <b$.  We know from  Proposition \ref{prop:d1} that $\mu_1(I_a)= b^{p-1}$ for
each $a\in \sD$.  Hence we can write
\begin{equation}
  S_w(1) = \sum_{a=1}^{a=b-1} b^{p-1} y_a,\qquad (a+1)^{-1} b\leq y_a \leq a^{-1}b\;.
\end{equation}
And we also have regarding $\int b^p \dx/x$: 
\begin{equation}
  b^p \log(b) =  \sum_{a=1}^{a=b-1} b^{p-1} z_a,\qquad (a+1)^{-1} b\leq z_a \leq a^{-1}b\;.
\end{equation}
Hence
\begin{equation}
  \left| S_w(1) - b^p \log(b) \right| \leq  \sum_{a=1}^{a=b-1}
  b^{p}(\frac1a-\frac1{a+1}) = (b-1)b^{p-1}\;.
\end{equation}
This proves the case $k=1$ for Theorem \ref{thm:limit}.

Let $k=2$.  The exact same decomposition we employed for $k=1$ works too here:
indeed $\mu_2(I_a) = b^{p-1}$ from Theorem \ref{thm:totalmass} applied to the
one-digit strings $s=a$, $1\leq a < b$, for which the condition $k> k_w^*(s)$ is
certainly valid as $k_w^*(s)\leq 1$.  So the case $k=2$ of Theorem \ref{thm:limit} is done.

Let $k>2$.  We partition $[b^{-1},1)$ into intervals $I(s) =
\bigl[n(s)/b^{k-1},(n(s)+1)/b^{k-1}\bigr)$ over all strings $s$ of length
$k-1$ and whose first digit is non-zero ($n(s)$ is the associated integer).
The value of $\mu_k(I(s))$ is the total mass of all $k$-admissible strings
having $s$ as prefix (for example with $k=4$ and whatever $w$, $10$
can not contribute any mass to $[0.100,0.101)$ despite $x(10)$ belonging
to it).  By Theorem \ref{thm:totalmass}, as $k$ is greater than the length of
$s$, this total mass is $b^p\cdot b^{-|s|} = b^{p-k+1}$.
We can thus simultaneously
write $S_w(k) = \sum_{s} b^{p-k+1} y_s$ with $(n(s)+1)^{-1}b^{k-1}\leq y_s\leq
n(s)^{-1}b^{k-1}$ and $b^p \log(b) = \sum_{s} b^{p}\cdot b^{-k+1} z_s$ with
$(n(s)+1)^{-1}b^{k-1}\leq z_s\leq n(s)^{-1}b^{k-1}$, hence
\begin{equation}
  \left| S_w(k) - b^p \log(b) \right| \leq  \sum_{s}
  b^{p-k+1}(\frac{b^{k-1}}{n(s)}-\frac{b^{k-1}}{n(s)+1}) =
  b^{p-k+1}\Bigl(\frac{b^{k-1}}{b^{k-2}}-\frac{b^{k-1}}{b^{k-1}} \Bigr)\;,
\end{equation}
and the proof of Theorem \ref{thm:limit} is complete.

Let us now prove the weak convergence of the $\mu_k$'s to $b^p$ times Lebesgue
measure.  It will suffice to prove that for any $0\leq t < u \leq 1$ such that
there exists an integer $l$ with $b^lt$ and $b^l u$ both integers then
$\mu_k([t,u)) = b^p(u-t)$ for $k$ large enough.  Indeed this implies $\lim
\mu_k(I) = b^p |I|$ for any interval $I\subset[0,1)$, where $|I|$ is Lebesgue
measure, as one sees from bounding $\liminf$ and $\limsup$ via contained (if
$|I|>0$) or containing intervals $[t,u)$.  See \cite[Prop. 26]{burnolirwin} for
the $p=1$ case.  In \cite{burnolirwin} the stabilization $\mu_k([t,u)) =
b(u-t)$ was proven for $k\geq l$, here we prove that $\mu_k([t,u)) = b^p(u-t)$
holds at least for $k\geq l+1$.  It is enough to handle $t = n/b^l$, $u
=(n+1)/b^l$, $0\leq n < b^l$, $n$ integer.  Let $s$ be the length $l$ string
representing the integer $n$ with possibly added leading zeros.  The
$k$-admissible strings $X$ mapping to the interval $[t,u)$ all share $s$ as
common length $l$ prefix (we used that $X$ has surely length $>l$ if it is
to be $k$-admissible, as $k>l$).  Theorem \ref{thm:totalmass} implies (see the
sentence following it) that for $k>l$ the
total mass of such $k$-admissible strings with common prefix $s$ is
$b^{p-l}$.  This means that $\mu_k([t,u))=b^p(u-t)$ for such $t = n/b^l$, $u
=(n+1)/b^l$, $0\leq n < b^l$, $n$ integer.

The proof of Theorem \ref{thm:conv} is thus complete. 

\singlespacing

\providecommand\bibcommenthead{}
\def\blocation#1{\unskip}
\def\arxivurl#1{\href{https://arxiv.org/abs/#1}{\textsf{arXiv:#1}}}




\end{document}